\documentclass[11pt]{article} 
\usepackage{theorem}
\usepackage{latexsym}
\usepackage{amssymb}
\usepackage{amsmath}
\usepackage{amsfonts}
\usepackage{amscd}
\usepackage{mathrsfs}
\usepackage{bezier}
\usepackage{color}
\usepackage{eufrak}
\usepackage{epic}
\usepackage[all]{xy}
\usepackage{bbm}
\usepackage{enumerate}

\makeatletter

\newbox\authrun
\newtoks\authorrunning
\newtoks\tocauthor
\newbox\titrun
\newtoks\titlerunning
\newtoks\toctitle

\def\institute#1{\gdef\@institute{#1}}
\def\author#1{\gdef\@author{#1}}
\def\maketitle{
\begin{center}
{\LARGE\bf \@title \par}
\vskip 2em
\end{center}
\let\maketitle=\relax}

\def\section{
\@startsection{section}
{1}{\z@}
{3.5ex plus 1ex minus .2ex}
{2.3ex plus .2ex}
{\bf}}

\def\ps@run{%
 \def\@oddfoot{}\def\@evenfoot{}
 \def\@oddhead{\hfil \the\titlerunning/ \the\authorrunning\hfil \thepage }%
 \def\@evenhead{\hfil \hfil \thepage }}%

\def\footnoterule{\kern-3\p@
  \hrule width .4\columnwidth
  \kern 2.6\p@}
\long\def\@makefntext#1{\parindent 1em\noindent
            \hbox {\hss$\m@th^{\@thefnmark}$}#1}
            
\makeatother
\newcommand{\Def}{{\mathrm{Def}}}

\newcommand{\bz}{{\mathbb{Z}}}

\newcommand{\Q}{{\mathbb{Q}}}

\renewcommand{\hom}{{\mathrm{Hom}}}

\newcommand{\mcalC}{{\mathcal C}}

\newtheorem{theorem}{Theorem}[section]
\newtheorem{proposition}[theorem]{Proposition}
\newtheorem{lemma}[theorem]{Lemma}
\newtheorem{corollary}[theorem]{Corollary}
\theorembodyfont{\rmfamily}

\newtheorem{notation}{\mdseries{\itshape{Notation}}}
\numberwithin{equation}{section}
\begin{document}
\pagestyle{run}
\thispagestyle{plain}
\title{Deflation map and the sum of inverses of the element orders in finite groups} 
\institute{Department of Mathematics, Kindai University,
Higashi-Osaka, 577-8502,
Japan \\ \rm E-mail: odaf@math.kindai.ac.jp}
\titlerunning{Deflation map}
\authorrunning{Fumihito Oda}
\maketitle
\vskip 2em
\centerline{\Large\bf Fumihito Oda\!\,\footnotemark}
\vskip 1em
\centerline{\small\itshape\lineskip .75em 
Department of Mathematics, Kindai University, Japan} 
\centerline{\small\lineskip .75em 
{\itshape E-mail:} odaf@math.kindai.ac.jp}
\vskip 2em
\begin{abstract}
Let $G$ be a finite group. 
In this note, we construct an element ${\sigma}^G$ in the Burnside ring of $G$ over $\Q$, which gives a rational number $m(G)$ that is the sum of the inverses of the element orders in $G$, 
by using the deflation map $\Def^G_{G/N}$ induced by the $(G/N,G)$-biset $G/N$ for a normal subgroup $N$ of $G$.
As a corollary to Theorem, we obtain another expression for the rank of the crossed Burnside ring $B^{\rm c}(G)$ of $G$ 
and the determinant of the Cartan matrix of $p$-local Mackey algebra of $G$ over a field $k$ of characteristic $p>0$ with big enough.
\end{abstract}

\footnotetext{
{\rm This work was supported by JSPS KAKENHI Grant Number JP19K03457 and the Research Institute for Mathematical Sciences, an International Joint Usage/Research Center located in Kyoto  University.}\\
2010 \it Mathematics Subject Classification. {\rm Primary 19A22; Secondary 
16U60.}
\par\noindent\itshape Keywords. {\rm finite group, Burnside ring, biset functor, deflation map. }
}
%
%
\section{Introduction}

Let $G$ be a finite group with the identity element $e$. 
Let $R$ be a commutative unital ring. 
The theory of biset functors, introduced by Serge Bouc, gives 
a unified construction of the elementary operations of restriction, induction, inflation, deflation and transport by isomorphism \cite{Bouc00, Bouc06, Bouc07, Bouc10, BoucThevenaz00}. 
Let $B(G,H)$ be the Grothendieck group of $(G,H)$-bisets, where $G$ and $H$ are finite groups. 
The biset category $R\mcalC$ over $R$ is the category of finite groups with morphism sets $\hom_{R\mcalC}(H,G)=RB(G,H)$. 
The composition of morphisms is induced by the usual tensor product of bisets. A {\it biset functor} over $R$ is an $R$-linear functor from $R\mcalC$ to the category $R$-Mod of $R$-modules (\cite[3.2.2]{Bouc10}). 
Biset functors over $R$ form an abelian category, where morphisms are natural transformations of functors. 
Let $B$ be the Burnside biset functor.
The cardinality of a set $S$ is denoted by $|S|$.
We will use the algebra homomorphism
\[
\Q B(G) \ni \alpha\mapsto |\alpha|\in \Q
\]
extending the map $X \mapsto |X|$ sending a finite $G$-set $X$ to its cardinality $|X|$.
The purpose of this note is to construct an element $ {\sigma}^G\in \Q B(G)$, which gives a rational number $m(G)$ that is the sum of the inverses of the element orders in $G$ (see \cite{Baniasad Azad Khosravi and Rashidi} for instance), by using 
the deflation map $\Def^G_{G/N}$ which are tools used in some proofs of the biset theory introduced by Bouc.
We recall that, for $N \lhd G$, the {\it deflation map} of the Burnside ring is an additive morphism
\[
 \Def^G_{G/N} :  \Q B(G) \longrightarrow \Q B(G/N)
\]
that sends a $G$-set $X$ to the set of orbits $N\backslash X$ of $N$ on $X$ (\cite[p.78]{Bouc10}). 
This set has a natural structure of $G/N$-set. 
We denote by $s_G$ (resp. $\mcalC_G$) the set of conjugacy classes of (resp. cyclic) subgroups of $G$, and a set of representatives of $s_G$ (resp. $\mcalC_G$) is denoted by $[s_G]$ (resp. $[\mcalC_G]$). 
Let $p$ be a prime number. 
We denote by $s_{p}(G)$ (resp. $G_{p'}$) the set of $p$-subgroups (resp. $p'$-elements) of $G$, and a set of representatives of conjugacy classes of $s_{p}(G)$ (resp. $G_{p'}$) is denoted by $[s_{p}(G)]$ (resp. $[G_{p'}]$). 
We denote by $\overline{N}_G(H)$ the quotient group of $N_G(H)/H$ for $H\le G$. 
We denote by $[H, K]$ the subgroup of $G$ generated by all commutators $[h, k] = h^{-1}k^{-1}hk$ for elements $h \in H$ and $k \in K$. 
We write 
\[
 m(G)=\sum_{g\in G}\frac{1}{o(g)}
\]
which is the sum of the inverses of the element orders in $G$. 
The value has recently begun to be studied by various researchers (for instance \cite{Baniasad Azad Khosravi and Rashidi}).
We introduce an element
\[
{\sigma}^G=\sum_{H\in [s_G]}|G:H|{e}^G_H \in \Q B(G)\]
which is inspired by the Frobenius-Wielandt morphism (\cite{DressSiebeneicherYoshida92}, \cite{Romero}).
Let 
\[
 G\backslash {\alpha} = \left|{\rm Def}^G_{G/G}({\alpha})\right|
\]
for $\alpha\in \Q B(G)$.
The following theorem, obtained as a very simple consequence of Bouc's theory, is the main result of this note.
\bigskip

\noindent{\bf Theorem \ref{Thm:Def formula in KR(G)}}
{\it 
Let $G$ be a finite group. Then 
\[
 G\backslash {\sigma^G}=m(G).
\] 
}

\noindent
While the report is written in Japanese, Yuya Kojima presents a highly generalized and sophisticated take on theorem above (\cite{kojima24}, \cite{kojima25}).
As a corollary to Theorem \ref{Thm:Def formula in KR(G)}, we obtain another expression for the rank (\cite[Proposition 3.2]{OdaYoshida01}) of the crossed Burnside ring (\cite{Yoshida97}, \cite{OdaYoshida01}, \cite{Bouc03}, \cite{Balmer Dell'Ambrogio20}) $B^{\rm c}(G)$ of $G$ 
and the determinant (\cite[Theorem 2.8]{Bou11}) of the Cartan matrix (\cite[Theorem 2.8]{Bou11}) of $p$-local Mackey algebra $\mu_k(G,1)$ (\cite[12.3.2]{Bouc97}) of the Mackey algebra $\mu_k(G)$ (\cite[Section 3]{theve-webb95}) of $G$ over a field $k$ of characteristic $p>0$ with big enough as follows:
\medskip

\noindent{\bf Corollary \ref{cor_CBRandMac}
{\it 
 Let $G$ be a finite group. Then the following hold.
\begin{enumerate}
 \item The rank of the crossed Burnside ring $B^{\rm c}(G)$ of $G$ is equal to 
\[
\mathrm{rk}_{\bz}(B^{\rm c}(G)) = \sum_{K\in[{\it s}_G]}\frac{|C_G(K)|}{|N_G(K)|}|K|\left((K/[K,K])\backslash {\sigma^{K/[K,K]}}\right).
\]
\item The determinant of the Cartan matrix of the $p$-local Mackey algebra $\mu_k(G,1)$ is equal to 
\[
 \mathrm{det}\mathsf{C}(\mu_k(G,1)) =
\prod_{R\in [s_p(G)]}\prod_{s\in [{\overline{N}_G(R)}_{p'}]}\left(|{C_{\overline{N}_G(R)}(s)}|_p\left(R/[\langle sR\rangle,R]\backslash \sigma^{R/[\langle sR\rangle,R]}\right)\right).
\]
\end{enumerate}
}
}
\begin{notation}
We write $\mu_G$ the M\"obius function of the poset of subgroups of $G$. 
We write $\varphi$ the Euler's totien function.
\end{notation}

%
%
\section{Preliminaries}
In this section, we recall some well-known basic facts that will be used to prove the theorem stated in the next section.
\par\medskip
\begin{lemma}{\rm (\cite{Gluck81}, \cite{Yoshida83})}\label{Lem:IPformula} 
Let $G$ be a finite group. If $H$ is a subgroup of $G$, denote by $e^G_H$ the element of $\Q B(G)$ defined by 
\[
 e^G_H=\frac{1}{|N_G(H)|}\sum_{K\le H}|K|\mu_G(K,H)[G/K],
\]
where $\mu_G$ is the M\"{o}bius function for the poset of the subgroups of $G$.
Then $e^G_H = e^G_K$ if the subgroups $H$ and $K$ are conjugate in $G$, and the elements $e^G_H$, for $H\in [s_G]$, are the primitive idempotents of the $\Q$-algebra $\Q B(G)$.
\end{lemma}

\begin{notation}
We review a notation from \cite{Bouc10} and \cite{Romero}.
If $G$ is a finite group, and $N$ is a normal subgroup of $G$, denote by $m_{G,N}$ the rational number defined by
\[
 m_{G,N}=\frac{1}{|G|}\sum_{XN=G}|X|\mu_G(X,G).
\]
\end{notation}

\par\medskip
\begin{lemma}{\rm (\cite[5.2.4.4]{Bouc10})} \label{Lem:DefIdemp} 
Let $N$ be a normal subgroup of $G$.
If $H$ is a subgroup of $G$, then
\[
 {\rm Def}^G_{G/N}(e^G_H)=\frac{|N_G(HN)/HN|}{|N_G(H)/H|}m_{H,H\cap N}e^{G/N}_{HN/N}.
\]
In particular, 
\[
 {\rm Def}^G_{G/G}(e^G_H)=\frac{1}{|N_G(H)/H|}m_{H,H}e^{G/G}_{G/G}.
\]
\end{lemma}

\medskip

\par\medskip
\begin{lemma}{\rm (\cite[(5.6.1)]{Bouc10})}\label{Lem:MoebiusuInversion}
Let $G$ be a finite group. Then 
$m_{G,G}=0$ if $G$ is not cyclic, and $m_{G,G}=\varphi(n)/n$ if $G$ is cyclic of order $n$.
\end{lemma}

\section{Biset functors}
In this section we summarize some basic properties of the biset theory introduced by Bouc, based on \cite{Bouc10}. 
Denote by $G^{\rm op}$ the opposite group of $G$. 
If $G$ and $H$ are finite groups, then an $(H,G)$-biset is a left $(H\times G^{\rm op})$-set. 
As usual, for an $(H, G)$-biset $U$, we set $h\cdot u\cdot g := (h, g^{-1})u$, for any $h \in H, g \in  G, u \in U$.
If $U$ and $V$ are $(H,G)$-bisets, then a biset homomorphism $f$ from $U$ to $V$ is a map $f:U\to V$ such that $f(h\cdot u\cdot g)=h\cdot f(u)\cdot g$, for any $h\in H, u\in U, g\in G$.
Let $L$ be a finite group.
If $U$ is an $(H,G)$-biset, and $V$ is a $(K,H)$-biset, 
the composition of $V$ and $U$ is the set of $H$-orbits $V\times_HU$ on the cartesian product $V\times U$, 
where the right action of $H$ is defined by
$\forall (v,u)\in V\times U,\, \forall h\in H,\, (v,u)\cdot h=(v\cdot h, h^{-1}\cdot u)$. 
The $H$-orbit of $(v,u)\in V\times U$ is denoted by $(v,_Hu)$.
Then the set $V\times U$ has a $(L,G)$-biset structure defined by
\[
 \forall\ell\in  L,\forall (v,_Hu)\in V\times_H U,\forall g\in G, \ell\cdot (v,_H u)\cdot g=(\ell\cdot  v,_H u\cdot g).
\]
The biset Burnside group $B(H,G)$ is the Burnside group $B(H\times G^{\rm op})$, namely, the Grothendieck group of $(H,G)$-bisets.
If $N\trianglelefteq G$, the set $G/N$ is an $(G/N, G)$-biset, for the left action of $G/N$ by multiplication, and the right action of $G$ by projection to $G/N$, and then right multiplication in $G/N$. 
If $X$ is a $G$-set, then there is an isomorphism of $G/N$-sets
\[
 (G/N)\times_G X \cong  N\backslash X.
\]
So the {\em deflation map} ${\rm Def}^G_{G/N} : B(G)\to B(G/N)$ is such that ${\rm Def}^G_{G/N}([X]) = [N\backslash X]$,
for any finite $G$-set $X$ (\cite[p.78]{Bouc10}).

We write 
\[
 m(G)=\sum_{g\in G}\frac{1}{o(g)}
\]
which the sum of the inverses of the element orders in $G$. 
The value has recently begun to be studied by various researchers (for instance \cite{Baniasad Azad Khosravi and Rashidi}).
For simplicity of notation, let
\[
 G\backslash {\alpha} = \left|{\rm Def}^G_{G/G}({\alpha})\right|
\]
for $\alpha\in \Q B(G)$.

\medskip
As established in the proof of \cite[Theorem 6.2.5]{Bouc00}, the following relation holds:
\begin{lemma}\label{Lem:non-cyclic}
Let $G$ be a finite group and $H$ be a subgroup of $G$. 
Then the following are equivalent:
\begin{enumerate}
\item $H$ is non-cyclic. 
\item $G\backslash {e^G_H}=0$.
\end{enumerate}
\end{lemma}
{\itshape Proof.} 
Since Lemma \ref{Lem:DefIdemp} shows
\[
 G\backslash e^G_H = \left|{\rm Def}^G_{G/G}(e^G_H)\right| =\frac{1}{|N_G(H)/H|}m_{H,H}\left|e^{G/G}_{G/G}\right|=\frac{1}{|N_G(H)/H|}m_{H,H},
\]
Lemma \ref{Lem:MoebiusuInversion} yields the equivalence of (1) and (2).
 $\Box$
\par\medskip
We are now in a position to prove the theorem of this note.
\begin{theorem}\label{defg=mg}
\label{Thm:Def formula in KR(G)}
Let $G$ be a finite group. Then 
\[
 G\backslash {\sigma^G}=m(G).
\]
\end{theorem}
{\itshape Proof.} 
By applying the lemmas prepared so far, we can derive the following sequence of equalities:
\begin{align*}
 G\backslash {\sigma^G}
&=\sum_{H\in [s_G]}|G:H|{G\backslash}e^G_H\\
&=\sum_{H\in [\mcalC_G]}|G:H|{G\backslash}e^G_H \quad (\mbox{Lemma}\, \ref{Lem:non-cyclic})\\
&=\sum_{H\in [\mcalC_G]}|G:H|\left| {\rm Def}^G_{G/G}(e^G_H)\right| \quad \\
&=\sum_{H\in [\mcalC_G]}|G:N_G(H)|m_{H,H}\left|e^{G/G}_{G/G}\right|\quad (\mbox{Lemma}\, \ref{Lem:DefIdemp})\\
&=\sum_{H\in [\mcalC_G]}|G:N_G(H)|m_{H,H}\\
&=\sum_{H\in [\mcalC_G]}{|G:N_G(H)|}\frac{\varphi(|H|)}{|H|} \quad (\mbox{Lemma } \ref{Lem:MoebiusuInversion})\\
&=\sum_{\langle g \rangle\in [\mcalC_G]} 
{|G:N_G(\langle g \rangle)|}\frac{\varphi(o(g))}{o(g)}\\
&=\sum_{g\in G}\frac{1}{o(g)}\\
&=m(G).
\end{align*}
This completes the proof of the theorem. $\Box$

\medskip
Although several results exist based on representations using the sum of the inverses of the element orders in a group, this note aims to apply them to the two results below.

\begin{proposition}{\rm \cite[Proposition 3.2]{OdaYoshida01}}
Let $G$ be a finite group. Then the rank of the crossed Burnside ring $B^{\rm c}(G)$ of $G$ is equal to 
\[
\mathrm{rk}_{\bz}(B^{\rm c}(G)) = \sum_{K\in[{\it s}_G]}\frac{|C_G(K)|}{|N_G(K)|}|K|\left(\sum_{x\in K/[K,K]}\frac{1}{o(x)}\right).
\]
\end{proposition}

\begin{theorem}{\rm \cite[Theorem 2.8]{Bou11}}
Let $G$ be a finite group, let $p$ be a prime number, and let $k$ be a field of characteristic $p$, big enough to be a splitting field for all the groups $N_G(Q)/Q$, for $Q\in S_p(G)$. Then the determinant of the Cartan matrix of the algebra $\mu_k(G,1)$ is equal to 
\[
 \mathrm{det}\mathsf{C}(\mu_k(G,1)) =
\prod_{R\in S_p(G)}\prod_{s\in [{\overline{N}_G(R)}_{p'}]}\left(|{C_{\overline{N}_G(R)}(s)}|_p\sum_{x\in R/[\langle sR\rangle,R]}\frac{1}{o(x)}\right).
\]
\end{theorem}

We conclude this note by stating the following corollary of Theorem \ref{defg=mg}.
\begin{corollary}\label{cor_CBRandMac}
 Let $G$ be a finite group. Then the following hold.
\begin{enumerate}
 \item The rank of the crossed Burnside ring $B^{\rm c}(G)$ of $G$ is equal to 
\[
\mathrm{rk}_{\bz}(B^{\rm c}(G)) = \sum_{K\in[{\it s}_G]}\frac{|C_G(K)|}{|N_G(K)|}|K|\left(K/[K,K]\backslash {\sigma^{K/[K,K]}}\right).
\]
\item The determinant of the Cartan matrix of the algebra $\mu_k(G,1)$ is equal to 
\[
 \mathrm{det}\mathsf{C}(\mu_k(G,1)) =
\prod_{R\in S_p(G)}\prod_{s\in [{\overline{N}_G(R)}_{p'}]}\left(|{C_{\overline{N}_G(R)}(s)}|_p \left((R/[\langle sR\rangle,R])\backslash {\sigma^{R/[\langle sR\rangle,R]}}\right) \right).\]
\end{enumerate}
\end{corollary}
\section*{Acknowledgments}
The author would like to express his sincere gratitude to Professor Naoki Chigira for providing valuable information regarding the title of the manuscript \cite{Baniasad Azad Khosravi and Rashidi}. 
 
\end{document}